\documentclass[11pt]{amsart}

\usepackage[T1]{fontenc}
\usepackage[utf8]{inputenc}
\usepackage{lmodern}
\usepackage{microtype}
\usepackage{amsmath,amssymb}
\usepackage[colorlinks=true,linkcolor=blue,citecolor=blue,urlcolor=blue]{hyperref}

\newtheorem{theorem}{Theorem}[section]
\newtheorem{lemma}[theorem]{Lemma}

\theoremstyle{remark}

\newcommand{\Q}{\mathbb Q}
\newcommand{\R}{\mathbb R}
\newcommand{\Z}{\mathbb Z}
\DeclareMathOperator{\ord}{ord}

\title[Neighbouring denominators and Erd\H{o}s--Mahler]
{A Neighboring-Denominator Variant of the Erd\H{o}s--Mahler Conjecture}

\author{Diego Marques}
\address{Departamento de Matem\'atica, Universidade de Bras\'ilia, Bras\'ilia, 70910-900, Brazil}
\email{diego@mat.unb.br}

\subjclass[2020]{Primary 11J70; Secondary 11J86, 11A51}
\keywords{continued fractions, Liouville numbers, smooth numbers, prime factors, $p$-adic logarithmic forms}

\begin{document}

\begin{abstract}
We prove a quantitative neighboring-denominator variant of the
Erd\H{o}s--Mahler conjecture. Let $p_n/q_n$ be the convergents of an
irrational real number $\xi$. If $p_nq_nq_{n+1}$ is $S$-smooth for
infinitely many $n$, where $S$ is a fixed finite set of primes, then
there exists an effectively computable constant $c=c(S)>0$ such that
\[
\log q_{n+1}\gg_{\xi,S} q_n^c
\]
along those indices. Consequently, $\xi$ is a Liouville number. The
proof uses the determinant identity for consecutive convergents and a
fixed-base consequence of Yu's theorem on $p$-adic logarithmic forms.
\end{abstract}

\maketitle

\section{Introduction}

Let
\[
 \xi=[a_0;a_1,a_2,\ldots]\in\R\setminus\Q
\]
and let $p_n/q_n$ denote its convergents, with $q_n>0$. For a nonzero
integer $N$, let $P(N)$ denote the largest prime divisor of $|N|$, with
$P(\pm1)=1$.

Erd\H{o}s and Mahler \cite{ErdosMahler1939} initiated the study of prime
divisors of continued-fraction convergents. In particular, they proved
that if
\begin{equation}\label{eq:EM-three}
 P(q_{n-1}q_nq_{n+1})
\end{equation}
is bounded for infinitely many $n$, then $\xi$ is a Liouville number. They further conjectured that the condition
\begin{equation}\label{eq:EM-conj}
 P(p_nq_n)\ \text{is bounded for infinitely many }n
\end{equation}
has the same consequence. We shall refer to this as the Erdős--Mahler conjecture. Related results were obtained by Fraenkel
\cite{Fraenkel1962,Fraenkel1964} and Lelis and Marques \cite{LelisMarques2017}. 

In 1983, Shorey \cite[p.~132]{Shorey1983} proved that
\[
 P(q_{n-1}p_nq_n)\longrightarrow\infty.
\]
This settles the case of the preceding denominator. The situation for the
next denominator is subtler: bounded prime support of $p_nq_nq_{n+1}$
need not itself yield a contradiction, because $q_{n+1}$ may grow very
rapidly. Our main result identifies this growth as the decisive mechanism
and shows that it already forces $\xi$ to be Liouville. 

\begin{theorem}\label{thm:main}
Let $\xi$ be an irrational real number, and let $p_n/q_n$ denote its
convergents. Suppose that $p_nq_nq_{n+1}$ is $S$-smooth for infinitely
many $n$, for some fixed finite set of primes $S$. Then there exist
$c=c(S)>0$ and $C=C(\xi,S)>0$ such that
\[
\log q_{n+1}\ge Cq_n^c
\]
for every such $n\ge2$. Consequently,
$|q_n\xi-p_n|<\exp(-Cq_n^c)$ at these indices, and in particular $\xi$
is a Liouville number.
\end{theorem}

In the terminology of \cite[Definition~2.1]{AriasKirilovMedeira2019},
Theorem~\ref{thm:main} in fact shows that $\xi$ is an exponential
Liouville number of order $1+\lceil1/c\rceil$. Its qualitative
conclusion already follows from Ridout's theorem \cite{Ridout1957}.
Indeed, if $q_{n+1}\le q_n^m$ along infinitely many relevant indices,
the determinant identity, applied to
$A_n=|p_n|q_{n+1}$ and $B_n=|p_{n+1}|q_n$, gives
$|A_n-B_n|=1$; the $S$-smoothness of $A_n$ and the polynomial bound on
the part of $B_n$ supported outside $S$ then contradict Ridout's
theorem. Thus $q_{n+1}>q_n^m$ eventually for every $m$. The new content
of Theorem~\ref{thm:main} is the exponential lower bound above.

The hypothesis is nonvacuous. There exist irrational numbers $\xi$ for
which $p_nq_nq_{n+1}$ is $3$-smooth for infinitely many $n$. To see
this, construct the continued fraction recursively. The cylinder
determined by any finite prefix contains some $2^u/3^v$, by the density
of $\{2^u/3^v:u,v\ge1\}$; extend the prefix so that
$p_n/q_n=2^u/3^v$. Since consecutive denominators are coprime,
$\gcd(q_{n-1},3)=1$, and since $2$ is a primitive root modulo $3^v$,
we may choose arbitrarily large $w$ with
$2^w\equiv q_{n-1}\pmod{3^v}$. Taking
$a_{n+1}=(2^w-q_{n-1})/3^v$ gives $q_{n+1}=2^w$, and hence
$p_nq_nq_{n+1}=2^{u+w}3^v$. Iterating gives the claim.

We now turn to the proof of Theorem~\ref{thm:main}. Its main input is a
fixed-base $p$-adic linear-form estimate, applied to the determinant
identity for consecutive convergents.

\section{Auxiliary results}

We record here the two ingredients needed for the proof of the main
theorem: standard properties of continued-fraction convergents and a
$p$-adic estimate for smooth integers.

Throughout, the notation $X\ll_{\mathcal P}Y$, equivalently
$X=O_{\mathcal P}(Y)$, means that $|X|\le CY$ for some constant
$C>0$ depending only on the parameters $\mathcal P$; the subscript is
omitted when the constant is absolute.

\subsection{Continued-fraction identities}

We begin by recalling the standard identities for continued-fraction
convergents that will be used in the proof; see, for example,
\cite[Chapter~1]{Bugeaud2004}.

\begin{equation}\label{eq:det}
 p_{n+1}q_n-p_nq_{n+1}=(-1)^n,
\end{equation}
\begin{equation}\label{eq:coprime}
 \gcd(p_n,q_n)=1,
 \qquad
 \gcd(q_n,q_{n+1})=1,
\end{equation}
together with the classical estimate
\begin{equation}\label{eq:approx}
 0<
 \left|\xi-\frac{p_n}{q_n}\right|
 <
 \frac{1}{q_nq_{n+1}}.
\end{equation}

\subsection{A $p$-adic estimate for smooth integers}

Let $S$ be a finite set of primes. A nonzero integer is said to be
\textit{$S$-smooth} if all of its prime divisors belong to $S$. For
$p$ prime and $x\in\Q^\times$, we write $\ord_p(x)$ for the normalized
$p$-adic valuation, so that $\ord_p(p)=1$.

We shall use the following fixed-base consequence of Yu's theorem \cite[\S1.1, p.~190]{Yu2007}.

\begin{lemma}\label{lem:Yu-special}
Let $p,\ell_1,\ldots,\ell_r$ be fixed primes, with $\ell_i\ne p$
for $1\le i\le r$. There exists a constant $C>0$, depending only on
$p,r,\ell_1,\ldots,\ell_r$, such that, for all integers
$b_1,\ldots,b_r$ satisfying
\[
 \ell_1^{b_1}\cdots\ell_r^{b_r}\ne1,
\]
one has
\begin{equation}\label{eq:Yu-special}
 \ord_p\!\left(
 \ell_1^{b_1}\cdots\ell_r^{b_r}-1
 \right)
 \le
 C\log\left(2+\max_{1\le i\le r}|b_i|\right).
\end{equation}
\end{lemma}

The following consequence is the form that will be applied to the
denominators of continued-fraction convergents.

\begin{lemma}\label{lem:smooth-congruence}
Let $S$ be a finite set of primes. There exists a constant $C_S>0$
such that, whenever $A,B>1$ are coprime $S$-smooth integers satisfying
\begin{equation}\label{eq:divides-abstract}
 B\mid A^2-1,
\end{equation}
one has
\begin{equation}\label{eq:smooth-congruence}
 \log B\le C_S\log\log A
\end{equation}
for all sufficiently large $A$.
\end{lemma}

\begin{proof}
Let $p\in S$ be a prime divisor of $B$. Since $\gcd(A,B)=1$, we have
$p\nmid A$, and hence
\[
 A=\prod_{\ell\in S\setminus\{p\}}\ell^{b_\ell},
 \qquad b_\ell\in\Z_{\ge0}.
\]
We apply Lemma~\ref{lem:Yu-special} to the fixed primes
$\ell\in S\setminus\{p\}$ with exponents $2b_\ell$. Since $A>1$,
\[
 \prod_{\ell\in S\setminus\{p\}}\ell^{2b_\ell}
 =A^2\ne1,
\]
so the non-vanishing hypothesis is satisfied. Therefore,
\begin{equation}\label{eq:Yu-use}
 \ord_p(A^2-1)
 \ll_{S,p}
 \log\left(
 2+\max_{\ell\in S\setminus\{p\}} b_\ell
 \right).
\end{equation}
Moreover,
\[
 b_\ell\log\ell\le\log A
\]
for every $\ell\in S\setminus\{p\}$, and hence
\[
 \max_{\ell\in S\setminus\{p\}}b_\ell
 \le \frac{\log A}{\log2}.
\]
It follows that
\begin{equation}\label{eq:vp-loglog}
 \ord_p(A^2-1)\ll_{S,p}\log\log A
\end{equation}
for all sufficiently large $A$.

Since $B\mid A^2-1$, we have
\[
 \ord_p(B)\le \ord_p(A^2-1).
\]
Summing over the primes $p\in S$ gives
\[
 \log B
 =\sum_{p\in S}\ord_p(B)\log p
 \ll_S\log\log A,
\]
which proves the lemma.
\end{proof}

We are now ready to prove the main result.

\section{Proof of Theorem~\ref{thm:main}}

Let $S$ be a finite set of primes such that $p_nq_nq_{n+1}$ is
$S$-smooth along an infinite set of indices $\mathcal N$. Since $\xi$ is irrational,
$p_n\ne0$ for all sufficiently large $n$ in this set. Put
\[
 A_n=|p_n|q_{n+1},
 \qquad
 B_n=q_n.
\]
Then $A_n$ and $B_n$ are $S$-smooth. Moreover,
\eqref{eq:coprime} gives
\[
 \gcd(A_n,B_n)=1.
\]
By \eqref{eq:det},
\[
 p_nq_{n+1}\equiv\pm1\pmod{q_n},
\]
and therefore
\begin{equation}\label{eq:central-congruence}
 q_n\mid A_n^2-1.
\end{equation}
Lemma~\ref{lem:smooth-congruence} now yields
\begin{equation}\label{eq:logqn-loglogA}
 \log q_n\ll_S\log\log A_n.
\end{equation}
Consequently, there exist constants $c,c_1>0$, depending only on $S$,
such that
\begin{equation}\label{eq:logA-power}
 \log A_n\ge c_1q_n^c
\end{equation}
for all sufficiently large $n\in\mathcal N$.

Since $p_n/q_n\to\xi$, we have $p_n=O_{\xi}(q_n)$. Hence
\[
 \log q_{n+1}
 =\log A_n-\log|p_n|
 \ge c_1q_n^c-O_\xi(\log q_n).
\]

Therefore, we obtain
\begin{equation}\label{eq:quant}
 \log q_{n+1}\ge Cq_n^c
\end{equation}
for some $C=C(\xi,S)>0$ and all sufficiently large $n\in\mathcal N$.
Consequently, for every $m\ge1$,
\[
 q_{n+1}>q_n^m
\]
whenever $n\in\mathcal N$ is sufficiently large. Hence, by
\eqref{eq:approx},
\[
 0<
 \left|\xi-\frac{p_n}{q_n}\right|
 <
 \frac{1}{q_n^{m+1}}.
\]
Since $\mathcal N$ is infinite and $m$ is arbitrary, $\xi$ is a
Liouville number.
\qed

\section*{Acknowledgements}

This work was supported by the Brazilian National Council for Scientific and Technological Development (CNPq), Grant No.~304467/2023-5.

\section*{Declaration on the use of AI tools}
OpenAI's ChatGPT was used to assist
with English-language and stylistic editing of the manuscript. The author takes full responsibility for the manuscript.

\end{document}